\documentclass[12pt]{article}

\usepackage{mathtools}
\usepackage{amsmath,amsfonts,amssymb,amsthm}
\usepackage{bbm}
\usepackage{graphicx}
\usepackage{tabularx}
\usepackage{hyperref}
\usepackage{color}
\newcommand{\avint}{\int \hspace{-1.0em}-\,}

\definecolor{darkgreen}{rgb}{0,0.55,0}
\newcommand{\grad}{\nabla}

\newcommand{\laplace}{\Delta}

\renewcommand{\div}{\grad\cdot}

\newcommand{\F}{{\mathcal F}}

\newcommand{\R}{\mathbbm{R}}

\newcommand{\Z}{\mathbbm{Z}}

\DeclareMathOperator{\LP}{LP}
\newcommand{\elp}{E_{\varphi}^{\LP}}

\usepackage{sectsty}

\sectionfont{\large}
\subsectionfont{\normalsize}
\subsubsectionfont{\normalsize}
\paragraphfont{\normalsize}

\newcommand{\la}{\langle}
\newcommand{\ra}{\rangle}
\newcommand{\lla}{\langle\langle}
\newcommand{\rra}{\rangle\rangle}
\newcommand{\sign}{\mbox{\rm sign}\, }

\def\Xint#1{\mathchoice
{\XXint\displaystyle\textstyle{#1}}%
{\XXint\textstyle\scriptstyle{#1}}%
{\XXint\scriptstyle\scriptscriptstyle{#1}}%
{\XXint\scriptscriptstyle\scriptscriptstyle{#1}}%
\!\int}
\def\XXint#1#2#3{{\setbox0=\hbox{$#1{#2#3}{\int}$ }
\vcenter{\hbox{$#2#3$ }}\kern-.59\wd0}}
\def\avint{\Xint-}

\newtheorem{prop}{Proposition}
\newtheorem{theorem}{Theorem}
\newtheorem{lemma}{Lemma}

\begin{document}
\phantom{ }
\vspace{4em}

\begin{flushleft}
{\large \bf On the Littlewood--Paley spectrum for passive scalar transport equations}\\[2em]
{\normalsize \bf Christian Seis}\\[0.5em]
\small Institut f\"ur Analysis und Numerik,  Westf\"alische Wilhelms-Universit\"at M\"unster, Germany.\\
E-mail: seis@wwu.de\\[2em]
Date: \today\\[3em]
\end{flushleft}

\noindent
{\bf Abstract:} We derive  time-averaged $L^1$ estimates on  Littlewood--Paley decompositions for linear advection-diffusion equations. For wave numbers close to the dissipative cut-off, these estimates are consistent with Batchelor's predictions on the variance spectrum in passive scalar turbulent mixing.

\vspace{2em}

\section{Introduction}

\subsection{Model and main results}

In this short paper, our aim is to derive bounds on the Littlewood--Paley projections of solutions to linear advection-diffusion equations with rough velocity fields. These equations are of the form
\begin{equation}\label{1}
\partial_t \theta +u\cdot \grad \theta -\kappa \laplace \theta= 0,
\end{equation}
where $\theta = \theta(t,x)\in\R$ is a tracer (or ``passive scalar''), $u=u(t,x)\in\R^d$ is a given divergence-free velocity field,
\begin{equation}\label{2}
\div u=0,
\end{equation}
and $\kappa $ is the positive diffusivity constant. We neglect any boundary effects by supposing that the evolution takes place in a box $[0,L]^d$ with periodic boundary conditions.
We equip the problem with an initial condition, that is,
\[
\theta(0,\cdot) = \theta_0.
\]

For simplicity, we shall assume that the spatial integral of  the square of the fractional velocity gradient $\grad^s u$ is constant in time, or equivalently,
\begin{equation}\label{3}
\la  |\grad^s u(t) |^2\ra^{1/2} =   G_s,
\end{equation}
for some constant $G_s$, where $\la \cdot \ra = L^{-d}\int_{[0,L]^d} \cdot\, dx$ denotes the spatial average, and $s\in[0,1]$. Velocity constraints of this form are natural in industrial processes, where $  G_0^2$ is the kinetic energy and $G_1^2$ the power or viscous dissipation rate. The fractional Sobolev norm on the left-hand side is defined on the Fourier level by
\[
\la |\grad^s u|^2\ra = \sum_{m\in \frac{2\pi}L\Z^d} |m|^{2s} |(\F u)(m)|^2,
\]
where $\F u$ is the Fourier transform of $u$, whose definition will be recalled in \eqref{20} below. 

Clearly, mild regularity assumptions on $u$ (in general much weaker than those in \eqref{3}) and the periodic boundary conditions imply that \eqref{1} preserves the spatial average, i.e., $\frac{d}{dt}\la \theta\ra = 0$. We may thus choose $\theta$ with vanishing  spatial average without losing any generality. Likewise, a Galilean transformation allows the restriction to mean-zero velocity fields, that is, $\la u\ra = 0$.


Before stating our main result, we shall introduce the Littlewood--Paley decomposition of our scalar function $\theta $,   whose time-dependency we neglect for a moment. We start by recalling the Fourier transform.

The Fourier transform $\F\zeta$ of an integrable periodic function $\zeta$ on $[0,L]^d$ is defined by
\begin{equation}\label{20}
(\F\zeta)(m) = \avint_{[0,L]^d} \zeta(x)e^{-im\cdot x}\, dx\quad\mbox{for }m\in \frac{2\pi}L \Z^d.
\end{equation}
In this context, $m$ is usually referred to as wave number.
The Fourier transform $\F\phi$ of a Schwartz function $\phi$ on $\R^d$ is defined by
\[
(\F\phi)(\xi) = \frac1{(2\pi)^{d/2}}\int_{\R^d} \phi(x) e^{-i\xi\cdot x}\, dx\quad\mbox{for }\xi\in \R^d.
\]
Here, $\xi$ is the frequency. 

We now select a family of Schwartz functions $\{\phi_{\ell}\}_{\ell\in\Z}$ defined on $\R^d$ such that their Fourier transforms satisfy
\begin{align}
(\F \phi_0)(\xi) &\not =0\quad\mbox{only if } |\xi|\in \left(2^{-1},2\right),\label{5}\\
(\F\phi_{\ell})(\xi) &  = (\F \phi_0)(2^{-\ell}\xi)\quad\mbox{for all $\xi$ and $\ell$},\label{6}\\
\sum_{\ell\in\Z} (\F\phi_{\ell})(\xi) &=1\quad\mbox{for any }\xi\not=0\label{7}.
\end{align}
The Littlewood--Paley decomposition $\{\theta_{\ell}\}_{\ell\in \Z}$ of   $\theta$   is then defined by
\[
\theta_{\ell}: = \phi_{\ell}\ast \theta,
\]
where the operation ``$\ast$'' is the convolution in space. We refer to $\theta_{\ell}$ as the Littlewood--Paley projection of $\theta$ at frequency $|\xi|\sim 2^{\ell}$.

Our main result provides an $L^1$ estimate for the Littlewood--Paley projections of $\theta$. It involves weighted long-time averages $\lla\cdot\rra_{\varphi} = \limsup_{T\to \infty}\frac1T\int_0^T\la \cdot\ra\, e^{\varphi (t)}dt$ for some positive increasing function $\varphi = \varphi(t)$.

\begin{theorem}\label{T2}
For any positive increasing function $\varphi=\varphi(t)$ it holds that
\begin{equation}\label{21}
\lla |\theta_{\ell}|\rra_{\varphi}\lesssim  \kappa^{-1}\left(2^{-(s+2)\ell}\la |\grad^s u|^2\ra^{1/2} + 2^{-3\ell}\|\frac{d\varphi}{dt}\|_{\infty}\right)\lla |\grad\theta|^2\ra^{1/2}\ra_{\varphi}.
\end{equation}

\end{theorem}

The estimate gives a bound on the $L^1$ norm of the Littlewood--Paley projections of the tracer variable  in terms of the velocity gradient and the average dissipation rate. It is  obvious that this estimate is optimal in the case of no stirring, $u=0$. Whether this estimate is mathematically sharp for certain no-trivial mixing flows is not clear to the author. 
In the following subsection, however, we will comment on the (weak) significance of this estimate for the mathematical theory of passive tracer turbulent mixing.

We have chosen the $L^1$ norm of the Littlewood--Paley projections over other Lebesgue norms in order to be able to estimate the nonlinearity (or, more precisely, the commutator of advection and Littlewood--Paley projection) against the $L^2$ norms of the velocity and the dissipation, which  have both physical meaning. Moreover, the inclusion of time weights is necessary in order to compensate the dissipation to zero in the long-time average. This can be already  seen on the level of the purely diffusive equation, where  $(\F \theta)(t,k)  = e^{-\kappa |k|^2 t} (\F\theta_0)(k)$ for every wavenumber $k$. Here $\varphi(t)  = \kappa |k_0|^2t$ for the smallest relevant wavenumber $k_0$ would be an appropriate choice. 

We remark that our  analysis of Littlewood--Paley projections is modelled after existing similar estimates in the context of the two- and three-dimensional Navier--Stokes equations by Constantin \cite{Constantin97} and Otto and Ramos \cite{OttoRamos10} and for the temperature distribution in  Rayleigh--B\'enard convection by the author \cite{Seis13a}.

 \subsection{Physical interpretation}

The linear advection-diffusion equation \eqref{1} describes the evolution of a scalar quantity $\theta$ that is simultaneously transported by the flow of the  vector field $u$ and  diffused at rate $\kappa$. We interpret the vector field $u$ as the velocity of an incompressible fluid, cf.~\eqref{2}, and $\theta$ is a tracer marker or a physical quantity.

If the flow is sufficiently turbulent, mixing of trace markers and physical quantities is a ubiquitous phenomenon. It can be observed in various areas of fluid dynamics, for instance, the mixing of saltwater and fresh water in estuaries or  the dispersion of pollutants in the earth's atmosphere. Besides their relevance in nature, mixing flows are of fundamental importance in numerous applications in industrial process engineering. Their theoretical study has been a major focus of research for many years; it has been  frequently reviewed, see, e.g., \cite{Ottino90,ShraimanSiggia00,Thiffeault12}. Flow mediated mixing is frequently referred to as stirring.

In the past years, fluid mixing attracted a remarkable attention by the  mathematical fluid dynamics communities and beyond. The majority  of the rigorous works, however, addressed the purely advective model, for instance, with a focus on absolute lower bounds on mixing rates  \cite{CrippaDeLellis08,LinThiffeaultDoering11,Lunasin12,Seis13b,IyerKiselevXu14}, optimal mixing strategies \cite{LinThiffeaultDoering11,Lunasin12,AlbertiCrippaMazzucato14,AlbertiCrippaMazzucato16,YaoZlatos17}, or universal mixers 
\cite{ElgindiZlatos18}. In the diffusive setting, it was shown that mixing flows enhance diffusive relaxation  \cite{ConstantinKiselevRyzhikZlatos08,BedrossianCotiZelati17,CotiZelatiDegadinoElgindi18}, while diffusion itself slows down the mixing rates \cite{MilesDoering17}.

In 1959, Batchelor analyzed the variance transfer from lower to higher frequencies that accompanies  the creation of gradients of $\theta$  by the turbulent fluid motion \cite{Batchelor59}. He predicts that for wave numbers in the so-called advective subrange $k\ll k_B$, the variance spectrum scales as
 \begin{equation}\label{8}
 E(k) \sim \chi \tau k^{-1}.
 \end{equation}
The advective subrange is the part of the equilibrium range for which the tracer's Fourier components  are (thought to be) independent of molecular diffusion. The Batchelor wave number $k_B$ is  inversely proportional to the Batchelor  dissipation scale at which stirring and diffusion balance, and it determines the large time decay rate of the tracer variance. To be more specific,  if, in a typical mixing scenario, the smallest length scales are reduced to the order of the Batchelor scale, the subsequent variance decay is essentially governed by the slowest diffusion rate, namely $e^{-2 \kappa k_B^2 t}$. This decay rate has been obtained for shear flows in \cite{BedrossianCotiZelati17} (modulo logarithmic corrections).


We will now show that our  main result, Theorem \ref{T2}, is consistent with the decay of the variance spectrum \eqref{8} for wave numbers that are \emph{of the order of the Batchelor wave number $k_B$}, if the normalizing factor $e^\varphi$ is chosen in such a way that it balances the variance   decay rate $e^{-2 \kappa k_B^2 t}$. This  rigorous result  thus  weakly connects a hypothesis on the large time mixing rate with the scaling of the variance spectrum. The interpretation applies to the case $s=1$ only.

Let us define the time-averaged Littlewood--Paley variance spectrum at frequency $k$ as
\begin{equation}\label{22}
\elp(k): = \frac{\lla |\theta_{\ell}|\rra_{\varphi}^2}{k}\quad\mbox{if }k\in[2^{\ell-1},2^{\ell}).
\end{equation}
Notice that this spectrum is sort of a (time-averaged) $L^1$ version of the traditional variance spectrum $E(k)$, which can be defined as  
\[
E(k) = \int_{|m|=k} |(\F \theta)(m)|^2\, dS(m).
\]
The variance decay rate is given by
$\chi  = - \frac{d}{dt}\la \theta^2\ra = \kappa \la |\grad\theta|^2\ra$.
As we expect  for large times that the variance  decays exponentially fast with rate $\kappa k_B^2$, we shall time-average $\chi$ and consider $\chi_{\varphi}=\la \chi^{1/2}\ra_{\varphi}^2$ with $d\varphi/dt\lesssim \kappa k_B^2 $.
Finally, in order to define the stirring time scale, we set $\tau = G_1^{-1} = \la|\grad u|^2\ra^{-1/2}$.
 Taking into account the constraint \eqref{3} on the velocity field, the Batchelor wave number $k_B$ is given by $k_B  = ({G_1}/{\kappa})^{1/2}$. Notice that for $s=1$, the stirring and diffusion time scales are of the same order. 

With these notations,  estimate \eqref{21} can be rewritten as
\[
\elp(k) \lesssim \left[\left(\frac{k_B}k\right)^{4+2s} + \left(\frac{k_B}k\right)^6 \right]\chi_{\varphi} \tau k^{-1},
\]
where $\varphi(t) \approx \kappa k_B^2 t$. Arguing as in \cite{Constantin97,OttoRamos10}, this implies that
\[
\elp(k) \lesssim  \chi_{\varphi} \tau k^{-1},
\]
for every $k\in [\beta k_B,\beta^{-1} k_B]$ with $\beta<1$ and a non-displayed constant dependent on $\beta$.
Therefore, one side of \eqref{8} holds in the last decades before the dissipative cut-off, if the variance spectrum is defined as in \eqref{22}. For the most interesting range of wave numbers less than $k_B$, no statement can be derived.

 The scaling of the Batchelor spectrum \eqref{8} is the passive scalar mixing analogue of Kolmogoroff's $k^{-5/3}$ law for the decay of the energy spectrum in the inertial subrange in turbulent flows \cite{Kolmogoroff41,Obukhoff41,Frisch95}. In fact, in mixing, the creation of filaments by the stirring velocity field can be interpreted as the transfer of tracer variance from small to large wave numbers, analogous to the energy transfer in turbulent flows in  the celebrated K41 theory. It is, however, by now commonly believed that the $-5/3$ power law is not exact. Responsible for  deviations are intermittency effects which seem to alter the numerical value of this exponent \cite{Frisch95,SreenivasanAntonia97,KanedaEtal03,DonzisSreenivasan10}.
 Nonetheless, there are attempts to approach the scaling of the energy spectrum rigorously, see, e.g.\ \cite{OttoRamos10}.

In contrast, Batchelor's $-1$ power law seems to be rather sturdy; even strong intermittency effects leave the law unchanged \cite{Kraichnan68}. Yet, the literature on this topic reports quite controversial experimental and computational results, see, for instance, \cite{DonzisSreenivasanYeong10} and the discussion therein.

%

%
%

We turn now to the proof of Theorem \ref{T2}.

\section{Proofs}


It will be necessary to localize $\theta$ on an even finer level (than $\theta_{\ell}$) in Fourier space. For this purpose, we cover the annulus $\{\xi\in \R^d:\: |\xi|\in (2^{\ell-1},2^{\ell+1})\}$ by a finite family of balls $\{B_{\sigma 2^{\ell}}(\xi_j)\}_{j=1,\dots,J }$, where $\sigma$ is a small positive number that will be fixed later,  and denote by $\{\psi_{\ell,j}\}_{j=1,\dots,J}$ a family of Schwartz functions whose Fourier transforms form a partition of unity subordinate to this covering. Notice that we can construct the $\psi_{\ell,j}$'s by scaling analogously to \eqref{6}, namely
\begin{equation}\label{16}
(\F\psi_{\ell,j}(\xi) = (\F\psi_{0,j})(2^{-\ell}\xi)\quad\mbox{for all $\xi$, $\ell$, and $j$.}
\end{equation}
 We then introduce a refinement of $\phi_{\ell}$ by setting $\phi_{\ell,j} = \phi_{\ell}\ast \psi_{\ell,j}$ and define
\[
\theta_{\ell,j} = \theta \ast \phi_{\ell,j} = \theta_{\ell}\ast\psi_{\ell,j}. 
\]

Our first result is a scale-by-scale energy estimate.

\begin{lemma}\label{L1}
There exists a universal constant $C>0$ such that
\begin{equation}\label{11}
\frac{d}{dt} \la |\theta_{\ell,j}|\ra +\frac{2^{2\ell} \kappa}{C} \la |\theta_{\ell,j}|\ra \le  \la |[u\cdot,\phi_{\ell,j}\ast]\grad\theta|\ra,
\end{equation}
where $[u\cdot,\phi_{\ell,j}\ast]$ is the commutator of the operations ``multiply by $u$'' and ``convolute with $\phi_{\ell,j}$''. 
\end{lemma}

\begin{proof}
We start by localizing the advection-diffusion equation \eqref{1} in Fourier space in the balls $B_{\sigma 2^{\ell}}(\xi_j)$,
\[
\partial_t \theta_{\ell,j} + u\cdot \grad\theta_{\ell,j} -\kappa\laplace \theta_{\ell,j} = [u\cdot,\phi_{\ell,j}\ast]\grad \theta.
\]
Here, we have used the fact that temporal and spatial derivatives commute with the operation $\phi_{\ell,j}\ast$.  Let $A(s)$ denote a smooth approximation of the modulus function $s\mapsto |s|$. An application of the chain rule then yields
\[
\partial_t A(\theta_{\ell,j}) + u\cdot \grad A(\theta_{\ell,j}) -\kappa A'(\theta_{\ell,j})\laplace \theta_{\ell,j} = A'(\theta_{\ell,j})[u\cdot ,\phi_{\ell,j}\ast]\grad\theta.
\]
Thanks to the periodic boundary conditions and the fluid's incompressibility encoded in \eqref{2}, the advection term on the left-hand side drops out when averaged over the cell $[0,L]^d$,
\[
\partial_t\la A(\theta_{\ell,j})\ra -\kappa \la A'(\theta_{\ell,j})\laplace\theta_{\ell,j}\ra = \la A'(\theta_{\ell,j})[u\cdot ,\phi_{\ell,j}\ast]\grad\theta\ra.
\]
(Notice that the original advection term still survives in the commutator term.) We will now carry out the approximation by choosing $A(s)=|s|$, which can be realized on a distributional level. We then obtain the estimate
\[
\partial_t\la |\theta_{\ell,j}|\ra -\kappa \la \sign(\theta_{\ell,j})\laplace\theta_{\ell,j}\ra \le \la |[u\cdot ,\phi_{\ell,j}\ast]\grad\theta|\ra.
\]

For the statement of the lemma, it remains to prove that
\begin{equation}\label{10}
-\la \sign(\theta_{\ell,j})\laplace\theta_{\ell,j}\ra \gtrsim 2^{2\ell} \la|\theta_{\ell,j}|\ra.
\end{equation}
For this purpose, we select a Schwartz function $\zeta $ whose Fourier transform is constantly $1$ on the unit ball, $(\F\zeta)(\xi)=1$ for $|\xi|\le 1$. We then define
\[
\zeta_{\ell,j}(x) = ( 2^{\ell}\sigma)^d \zeta( 2^{\ell} \sigma x) e^{i \xi_j\cdot x}
\]
and observe that $(\F\zeta_{\ell,j})(\xi) = (\F\zeta)\left(\frac{\xi-\xi_j}{2^{\ell}\sigma}\right)=1$ for $\xi\in B_{2^{\ell}\sigma}(\xi_j) $. As a consequence, $\zeta_{\ell,j}$ leaves $\theta_{\ell,j}$ invariant under convolution, $\theta_{\ell,j} = \theta_{\ell,j}\ast \zeta_{\ell,j}$. It follows that
\[
\laplace \theta_{\ell,j} + |\xi_j|^2\theta_{\ell,j} = \left(\laplace \zeta_{\ell,j}+|\xi_j|^2\zeta_{\ell,j}\right)\ast\theta_{\ell,j},
\]
and application of Young's convolution estimate then yields
\begin{equation}\label{13}
\la | \laplace \theta_{\ell,j} + |\xi_j|^2\theta_{\ell,j} |\ra \le \la |\theta_{\ell,j}|\ra\int_{\R^d} |\laplace \zeta_{\ell,j}+|\xi_j|^2\zeta_{\ell,j}|\, dx .
\end{equation}
We claim that
\begin{equation}\label{12}
\int_{\R^d} |\laplace \zeta_{\ell,j}+|\xi_j|^2\zeta_{\ell,j}|\, dx\lesssim 2^{2\ell}\sigma .
\end{equation}
Indeed, by a direct computation we find that
\[
\left(\laplace \zeta_{\ell,j} + |\xi_j|^2\zeta_{\ell,j} \right)(x) = \left((2^{\ell}\sigma)^{d+2} (\laplace \zeta)(2^{\ell}\sigma x) + 2(2^{\ell}\sigma)^{d+1} i\xi_j\cdot (\grad\zeta)(2^{\ell}\sigma x)\right)e^{i\xi_j\cdot x},
\]
and thus, integration and the change of variables $y = 2^{\ell}\sigma x$ yield
\[
\int_{\R^d} |\laplace \zeta_{\ell,j} + |\xi_j|^2\zeta_{\ell,j} |\, dx \lesssim 2^{2\ell} (\sigma^2+\sigma ) \int_{\R^d} |\laplace\zeta| + |\grad\zeta|\, dy.
\]
Because $\zeta$ is a Schwartz function and $\sigma $ small (say, smaller than $1$), we deduce \eqref{12}. It remains to plug \eqref{12} into \eqref{13} and conclude that
\begin{align*}
-\la \sign(\theta_{\ell,j})\laplace \theta_{\ell,j}\ra & = \la \sign(\theta_{\ell,j})|\xi_j|^2\theta_{\ell,j}\ra - \la \sign(\theta_{\ell,j})\left(\laplace \theta_{\ell,j} + |\xi_j|^2\theta_{\ell,j}\right)\ra\\
&\ge \left(|\xi_j|^2 - \frac{2^{2\ell}\sigma}C\right)\la |\theta_{\ell,j}|\ra,
\end{align*}
for some universal constant $C>0$. Using  $|\xi_j|\ge 2^{\ell-1}$ and choosing $\sigma$ sufficiently small implies \eqref{10}  as desired.
\end{proof}

The left-hand side in the energy estimate \eqref{11} is further bounded with the help of the following auxiliary convolution estimate.

\begin{lemma}\label{L2}
Suppose that $v = v(x) $ and $q=q(x) $ are $[0,L]^d$ periodic  functions with $\la |\grad^s v|^2\ra ,\la q^2\ra <\infty$ and $\phi = \phi(y)$ is a Schwartz function on $\R^d$, then
\[
\la |[v,\phi\ast]q|\ra \lesssim \left(\int_{\R^d}|\phi(y)| \, dy\right)^{1-s}\left(\int_{\R^d}|\phi(y)||y| \, dy\right)^s \la |\grad^s v|^2\ra^{1/2}\la q^2\ra^{1/2}.
\]
\end{lemma}

\begin{proof}Notice first that it is enough to consider the pivotal cases $s=0$ and $s=1$. The general case can be obtained via interpolation. Indeed, for $v\in H^s$ and an arbitrary $M>0$, we consider the decomposition $v=v_0^M+v_1^M$ with 
\[
(\F v_0^M)(m) = \begin{cases} (\F v)(m) &\mbox{if }|m|>M,\\ 0&\mbox{otherwise}.\end{cases}
\]
Then $v_0^M\in L^2$ and $v_1^M\in H^1$. If the statement is proved for $s=0$ and $s=1$, then 
\begin{align*}
\la |[v,\phi\ast]q|\ra & \le \la |[v_0^M,\phi\ast]q|\ra + \la |[v_1^M,\phi\ast]q|\ra\\ 
 &\le \left(\int_{\R^d}|\phi(y)|\, dy \la |  v_0^M|^2\ra^{1/2}+\int_{\R^d}|\phi(y)||y| \, dy \la |  \grad v_1^M|^2\ra^{1/2}\right)\la q^2\ra^{1/2}.
\end{align*}
From the definition of $v_0^M$ and $v_1^M$ it immediately follows that $ \la |  v_0^M|^2\ra^{1/2} \le M^{-s}\la |\grad^{s} v|^2\ra^{1/2}$ and $ \la | \grad v_1^M|^2\ra^{1/2} \le M^{1-s}\la |\grad^{s} v|^2\ra^{1/2}$ , and thus 
\[
\la |[v,\phi\ast]q|\ra \le  \left(M^{-s} \int_{\R^d}|\phi(y)|  \, dy+M^{1-s} \int_{\R^d}|\phi(y)| |y|\, dy  \right) \la | \grad^s v|^2\ra^{1/2} \la q^2\ra^{1/2}.
\]
Minimizing in $M$ yields the desired result.

We now turn to the estimate for $s=1$. The statement for the remaining case $s=0$ is actually simpler and shall be omitted here. We start with a pointwise statement. For any $x$, it holds that
\begin{align*}
[v,\phi\ast]q(x) & = \int_{\R^d}\phi(y)  (v(x)-v(x-y)) q(x-y)\, dy\\
& = \int_0^1 \int_{\R^d} \phi(y) y\cdot (\grad v)(x-sy) q(x-y)\, dyds.
\end{align*}
Averaging in $x$ and successively applying Fubini's theorem and H\"older's inequality yield
\begin{align*}
\MoveEqLeft\la | [v,\phi\ast]q|\ra  \le \int_0^1\int_{\R^d} \avint_{[0,L]^d} |\phi(y)| |\grad v (x-sy)|y||q(x-y)|\, dxdyds\\
&= \int_0^1 \int_{\R^d} |\phi(y)||y|\left(\avint_{[0,L]^d}|\grad u(x-sy)|^2\, dx\right)^{1/2}\left(\avint_{[0,L]^d} q(x-y)^2\,dx\right)^{1/2}dyds.
\end{align*}
It only remains to invoke the periodicity in $x$ to conclude the statement of the lemma.
\end{proof}

\begin{prop}\label{P1}
There exists a universal constant $C>0$ such that
\begin{equation}\label{14}
\frac{d}{dt} \la |\theta_{\ell,j}|\ra + \frac{2^{2\ell}\kappa}{C} \la |\theta_{\ell,j}|\ra \le C 2^{-s\ell} \la |\grad^s u|^2\ra^{1/2}\la |\grad \theta|^2\ra^{1/2} .
\end{equation}
\end{prop}

\begin{proof}
The statement is an immediate consequence of the previous two lemmas together with the observation that 
\begin{equation}\label{15}
\int_{\R^d} |\phi_{\ell,j}(y)||y|^r\, dy\lesssim 2^{-r\ell}
\end{equation}
for any real $r$. Our argument for \eqref{15} relies on the scaling assumptions in \eqref{6} and \eqref{16}. Indeed, the latter imply that $(\F\phi_{\ell,j})(\xi) = (\F\phi_0)(2^{-\ell}\xi)(\F\psi_{0,j})(2^{-\ell}\xi)$, so that via a change of variables,
\begin{align*}
\phi_{\ell,j}(y) & = \frac1{(2\pi)^{d/2} }\int_{\R^d} e^{i\xi\cdot y}(\F\phi_0)(2^{-\ell}\xi)(\F\psi_{0,j})(2^{-\ell}\xi)\, d\xi\\
& = \frac{2^{d\ell}}{(2\pi)^{d/2} }\int_{\R^d} e^{i\eta\cdot 2^{\ell} y}(\F\phi_0)(\eta)(\F\psi_{0,j})(\eta)\, d\eta\\
& = 2^{d\ell} (\phi_0\ast\psi_{0,j})(2^{\ell}y).
\end{align*} 
Therefore, applying a change of variables in real coordinates, we find that
\begin{align*}
\int_{\R^d} |\phi_{\ell,j}(y)||y|^r\, dy &= 2^{d\ell} \int_{\R^d} |(\phi_0\ast\psi_{0,j})(2^{\ell}y)||y|^r\, dy\\
&= 2^{- r\ell} \int_{\R^d} |(\phi_0\ast\psi_{0,j})(z)||z|^r\, dz.
\end{align*}
The integral is independent of $\ell$ and bounded by the virtue of the decay properties of Schwartz functions. This concludes the proof.
\end{proof}

We are now in the position to prove Theorem \ref{T2}.

\begin{proof}[Proof of Theorem \ref{T2}]
Our starting point is the differential inequality derived in Proposition \ref{P1} above. We smuggle the factor $e^\varphi$ into \eqref{14},
\[
\frac{d}{dt}\left(e^\varphi \la |\theta_{\ell,j}|\ra\right) + \frac{2^{2\ell}\kappa e^{\varphi}}C \la|\theta_{\ell,j}|\ra \le C2^{-s\ell}e^\varphi \la |\grad^s u|^2\ra^{1/2} \la |\grad\theta|^2\ra^{1/2}  + \frac{d\varphi}{dt} e^{\varphi} \la|\theta_{\ell,j}|\ra,
\]
and integrate in time over the interval $[0,T]$,
\begin{align*}
\MoveEqLeft e^{\varphi(T)} \la |\theta_{\ell,j}(T)|\ra + \frac{2^{2\ell} \kappa}C \int_0^T e^\varphi \la |\theta_{\ell,j}|\ra\,dt\\
& \le C2^{-s\ell}\la |\grad^s u|^2\ra^{1/2}\int_0^T e^\varphi \la |\grad \theta |^2\ra^{1/2}\, dt +\int_0^T \frac{d\varphi}{dt} e^{\varphi}\la |\theta_{\ell,j}|\ra\, dt + e^{\varphi(0)}\la |\theta_{\ell,j}(0)|\ra.
\end{align*}
Recall that we have chosen $u$ with a fixed  budget, so that $\la |\grad^s u|^2\ra$ is independent of time. Dropping the nonnegative first term on the left-hand side, passing to the long-time average and dividing by $2^{2\ell}$, we furthermore obtain
\begin{equation}\label{17}
\kappa \lla |\theta_{\ell,j}|\rra_{\varphi} \lesssim 2^{-(s+2)\ell}\la |\grad^s u|^2\ra^{1/2}  \lla |\grad\theta|^2\ra^{1/2}\ra_{\varphi} +2^{-2\ell}\|\frac{d\varphi}{dt}\|_{\infty}\lla |\theta_{\ell,j}|\rra_{\varphi}.
\end{equation}

Observe now that
\begin{equation}\label{18}
\la |\theta_{\ell,j}|\ra \lesssim 2^{-\ell} \la |\grad \theta_{\ell,j}|\ra,
\end{equation}
and
\begin{equation}\label{19}
\la|\grad\theta_{\ell,j}|\ra \lesssim \la |\grad \theta|^2\ra^{1/2}.
\end{equation}

The second estimate simply follows from Young's convolution estimate and Jensen's inequality,
\[
\la |\grad\theta_{\ell,j}|\ra \le \left(\int_{\R^d} |\phi_{\ell,j}|\, dy\right) \la |\grad \theta|\ra \lesssim \la |\grad\theta|\ra \le \la |\grad \theta|^2\ra^{1/2},
\]
where, as in the proof of Proposition \ref{P1}, 
\[
\int_{\R^d} |\phi_{\ell,j}|\, dy = \int_{\R^d} |\phi_{0,j}|\,dz \sim 1
\]
by \eqref{6} and because $\phi_{0,j}$ is a Schwartz function.

For the first estimate, \eqref{18}, we notice that in view of the scaling property \eqref{6}, it is enough to establish the statement of $\ell=0$. Due to the dyadic partition of unity of the frequency space in \eqref{5}--\eqref{7}, it holds that $\F \phi_{-1} + \F\phi_0 + \F\phi_1=1$ in the support of $\F\phi_0$. As a consequence, $\phi_{-1}+\phi_0+\phi_1$ leaves $\phi_{0,j}$ invariant under convolution. Therefore, for any $k\in \{1,\dots.d\}$,
\[
i\xi_k (\F \phi_{0,j})(\xi) = (\F\partial_{x_k} \phi_{0,j})(\xi) = \sum_{\ell=-1,0,1} (\F\phi_{\ell})(\xi) (\F\partial_{x_k}\phi_{0,j})(\xi),
\]
and thus
\[
(\F \phi_{0,j})(\xi) = \sum_{\ell=-1,0,1}\frac{\xi_k}{i|\xi_k|^2} (\F\phi_{\ell})(\xi) (\F\partial_{x_k}\phi_{0,j})(\xi).
\]
Recall that $(\F\phi_{\ell})(0)=0$ by the virtue of \eqref{5}, \eqref{6}. We now invoke Jensen's convolution estimate and find
\begin{align*}
\la |\phi_{0,j}|\ra & \le \la |\partial_k\phi_{0,j}|\ra \sum_{\ell=-1,0,1} \int_{\R^d} |\F^{-1}\left(\xi \mapsto \frac{\xi_k}{i|\xi_k|^2}(\F \phi_{\ell})(\xi) \right)|\, dx \\
&\lesssim \la |\partial_k\phi_{0,j}|\ra,
\end{align*}
where, in the second inequality, we have again used the fact that $\phi_{\ell}$ is a Schwartz function. 
\end{proof}

\section*{Acknowledgement}
The author acknowledges inspiring discussions with Charlie Doering on the Batchelor scale. He thanks the anonymous referees for helpful comments and suggestions.

\bibliography{mixing}
\bibliographystyle{acm}

\end{document}